\documentclass[11 pt]{article}

\usepackage[utf8]{inputenc}
\usepackage[english]{babel}   
\usepackage{graphicx}
\usepackage{lmodern}
\usepackage{tikz}
\usetikzlibrary{calc}
\usepackage{subfig}
\usepackage{longtable}
\usepackage{epsfig}
\usepackage{color}
\usepackage{setspace}
\usepackage[top=2cm,bottom=2.5cm,left=2.5cm,right=2.5cm]{geometry}
\usepackage{sectsty}
\usepackage{fancyhdr}
\usepackage[colorlinks=true]{hyperref}
\usepackage{caption}
\usetikzlibrary{patterns}
\usepackage{amsmath}
\usepackage{amsfonts}
\usepackage{amssymb}
\usepackage{amsthm}
\usepackage{algorithm}
\usepackage{algpseudocode}
\usepackage{siunitx}

\newtheorem{remark}{Remark}
\newtheorem{proposition}{Proposition}
\newtheorem{theorem}{Theorem}
\newtheorem{lemma}{Lemma}

\newcommand{\tr}{^{\mathsf{T}}}

\newcommand{\e}{\mathrm{e}}
\newcommand{\ATone}{\texttt{AT$_1$}}
\newcommand{\ATtwo}{\texttt{AT$_2$}}

\title{A DG/CR discretization for the variational phase-field approach to fracture}

\date{February 11, 2023}

\author{\begin{minipage}{\textwidth}\centering Fr\'ed\'eric Marazzato$^{1}$ and Blaise Bourdin$^{2}$\\
   \small{$^{1}$Department of Mathematics, Louisiana State University, Baton Rouge, LA 70803, USA}\\
   \small{$^{2}$Department of Mathematics \& Statistics, McMaster University, 1280 Main Street West, Hamilton, ON L8S-4K1, Canada}\\
   \small{email: \texttt{marazzato@lsu.edu ; bourdin@mcmaster.ca}}\end{minipage}}

\begin{document}
\maketitle

\abstract{
   Variational phase-field models of fracture are widely used to simulate nucleation and propagation of cracks in brittle materials.
   They are based on the approximation of the solutions of free-discontinuity fracture energy by two smooth function: a displacement and a damage field.
   Their numerical implementation is typically based on the discretization of both fields by nodal $\mathbb{P}^1$ Lagrange finite elements.
   In this article, we propose a nonconforming approximation by discontinuous elements for the displacement and nonconforming elements, whose gradient is more isotropic, for the damage.
   The handling of the nonconformity is derived from that of heterogeneous diffusion problems.
   We illustrate the robustness and versatility of the proposed method through series of examples.
} \newline

\textit{Keywords: Phase-field, Fracture, Discontinuous Galerkin}

\section{Introduction}
Variational phase-field models of fracture~\cite{Bourdin-1998,Bourdin-Francfort-EtAl-2000a,Bourdin-Francfort-EtAl-2008b,Bourdin-2007a}, are increasingly popular approaches to the numerical simulation of crack nucleation and propagation in brittle materials.
Their essence is the approximation of Francfort and Marigo's free discontinuity energy~\cite{francfort1998revisiting} by a functional depending on smooth variables: a continuous displacement field and a phase-field representing the cracks. 

The strengths of these methods are now well-established, particularly concerning their ability to handle complex crack paths, reproduce crack nucleation in multiple settings~\cite{tanne2018crack,Kumar-Bourdin-EtAl-2020a}, or extend to complex multi-physics problems~\cite{Maurini-Bourdin-EtAl-2012a,Bourdin-Marigo-EtAl-2014a,Brach-Tanne-EtAl-2019a,Chukwudozie-Bourdin-EtAl-2019a}.
As such, the vast majority of numerical implementations are based on functionals derived from that introduced in~\cite{Ambrosio-Tortorelli-1992} for the Mumford-Shah image segmentation problem (see also~\cite{Braides-1998a}), regularized using Lagrange finite elements~\cite{Bellettini-Coscia-1994,Bourdin-1999,Bourdin-Francfort-EtAl-2000a}.

Their main strength, the representation of the crack geometry in terms of a smooth field allowing to easily handle complex geometries can become a liability in some situations.
In the context of hydraulic, for example, fracturing accounting for crack aperture in the fluid flow simulations in the fracture system is particularly challenging~\cite{Yoshioka_2020}.
This issue has been typically tackled in an ad-hoc manner~\cite{Miehe-Mauthe-EtAl-2015a,Mikelic-Wheeler-EtAl-2015a,Wilson-Landis-2016a}, or through complex geometric reconstructions~\cite{Ziaei-Rad-Shen-EtAl-2016x}.
A rigorous treatment, based on non-local averages of the phase-field along lines intersecting fractures is available but computationally costly~\cite{Chukwudozie-Bourdin-EtAl-2019a}.
 
Combining the strength of phase-field models with an explicit representation of the crack geometry, using discontinuous finite elements, for instance could have significant advantages for complex coupled problems, but this idea has been left relatively unexplored, with the exception of~\cite{engwer2017phase} where a symmetric discontinuous Galerkin discretization for the displacement was introduced, while the phase-field variable remained discretized using nodal Lagrange elements and~\cite{muixi2020hybridizable} where a fully hybridized discontinuous Galerkin discretization is chosen for both displacement and damage.

In this article, we propose a novel approach using a non-symmetric discontinuous Galerkin discretization for the displacement and a Crouzeix--Raviart discretization for the damage.
In this approach, $\mathbb{P}^1$ discontinuous elements are used to approximate the displacements whose discontinuities are then localized along element faces or edges, so that the cracks, their orientation or normal aperture can easily be obtained.
The rationale for using $\mathbb{P}^1$ Crouzeix--Raviart finite elements to approximate the damage is that the gradient of Crouzeix--Raviart functions is more isotropic that classical $\mathbb{P}^1$ nodal finite elements as shown in~\cite{chambolle2020crouzeix}.
The choice of a non-symmetric version over a symmetric one is justified by its stability with respect to the value of the penalty parameter.
The robustness of the method regarding crack propagation and crack nucleation is asserted on numerical tests.
Because this discretization scheme allows displacement jumps along element faces, instead of along a long strip of elements, the proposed scheme leads to a better approximation of the surface energy, and reduces the needs to adjust for ``effective toughness'' as discussed in~\cite{Bourdin-Francfort-EtAl-2008b} (Section 8.1.1).

In Section~\ref{sec:continuous model}, we present the continuous model adopted to model crack propagation under quasi-static loading and then derive the evolution equations for the displacement and the damage. In Section~\ref{sec:discrete setting}, the discretization is introduced and proved to have discrete solutions.
The convergence towards the continuous model is proved in Appendix \ref{sec:proof}. In Section~\ref{sec:numerical tests}, numerical tests show the reliability of the method in the computation of crack propagation and crack nucleation in both one and two dimensions. Finally, a comparison with experimental results obtained in~\cite{pham2017experimental} is performed.
In Section~\ref{sec:conclusion}, some conclusions are drawn and the potential for future work is explored.

\section{Phase-field models of quasi-static brittle fracture}
\label{sec:continuous model}
Consider a brittle material occupying a bounded domain $\Omega$, a polyhedron in $\mathbb{R}^d$ ($d=2,3$) which can be perfectly fitted by simplicial meshes, with Hooke's law $\mathbb{C} = \lambda \mathbf{1} \otimes \mathbf{1} + 2\mu \mathbf{1}_4$, where $\mathbf{1}$ is the unit tensor of order 2,  $\mathbf{1}_4$ the unit tensor of order four, and $\lambda$ and $\mu$ are the Lamé coefficients. 
For any $x\in \Omega$, the displacement is written $u(x) \in \mathbb{R}^d$, the linearized strain is  $\e(u) = \frac12 \left(\nabla u + \nabla u^{\tr} \right) \in \mathbb{R}^{d \times d}_{\mathrm{sym}}$, and the stress is $\sigma(u) = \mathbb{C}\e(u) \in \mathbb{R}^{d\times d}_{\mathrm{sym}}$.
%The crack is written $\Gamma(t)$.
Let $\partial \Omega = \partial \Omega_D \cup \partial \Omega_N$ be a partition of the boundary of $\Omega$. We assume that  $\partial \Omega_D$ is  a closed set and $\partial \Omega_N$  is open in $\partial \Omega$.

We consider a set of discrete time steps $0 = t_0 < \dots < t_N = T$ and at each step $t_i$, a Dirichlet boundary condition $u_i=w_i$, with $w_i \in \left(H^{1/2}(\partial \Omega) \right)^d$, is imposed on $\partial \Omega_D$ while the remaining part of the boundary $\partial \Omega_N$ remain stress free.
%\FM{This restriction is assumed to ensure that the problem is well-posed \cite{Bourdin-Francfort-EtAl-2008b}.}
Let $\ell >0$ be a small regularization length, $\alpha \in A := \left\{ \alpha \in H^1(\Omega;[0,1]) ;\ \alpha=0 \text{ on } \partial \Omega_D\right\}$ the phase-field variable representing cracks, and 
\[
   U_i := \left\{ u \in \left(H^1(\Omega)\right)^d ; \ u=w_i \text{ on } \partial \Omega_D\right\}
\]
be the space of admissible displacements.
To any $(u,\alpha)\in V_i := U_i\times A$, we associate the generalized phase-field energy~\cite{Braides-1998a,Bourdin-Francfort-EtAl-2000a,Bourdin-Francfort-EtAl-2008b}
\begin{align}
   \label{eq:energy continuous Phase-field}
   \mathcal{E}_\ell(u,\alpha) := & \frac12 \int_{\Omega} \left(a(\alpha)+\eta_\ell\right) \mathbb{C}\e(u)  \cdot \e( u) \, dx + \frac{G_c}{c_w} \int_{\Omega} \left( \frac{w(\alpha)}{\ell} + \ell |\nabla \alpha|^2  \right)\, dx,\\
   & = \mathcal{E}^{(e)}_\ell(u,\alpha) + \mathcal{E}^{(s)}_\ell(u,\alpha ) \notag
\end{align}
where $a$ and $w$ are two non-negative $\mathcal{C}^1$ monotonic functions such that $a(0)=w(1) = 1$ and $a(1) = w(0) = 0$, $\eta_\ell = o(\ell)$ is a regularization parameter,  $c_w:= 4\int_0^1\sqrt{w(s)}\, ds$ is a normalization parameter, and $G_c$ is the  critical elastic energy release rate, a material property.
   The first term in~\eqref{eq:energy continuous Phase-field} approximates the elastic energy of the system while the second one represents the fracture energy, \emph{i.e.} in the limit of $\ell \to 0$ the aggregate crack length in two space dimensions and area in three dimensions, scaled by $G_c$. 
For the \ATone{} model subsequently used in this article, $a(\alpha) = (\alpha - 1)^2$ and $w(\alpha) = \alpha$.
Another classic choice is the \ATtwo{} model for which $a(\alpha) = (\alpha - 1)^2$ and $w(\alpha) = \alpha^2$.
%We also choose $a(\alpha) := (\alpha - 1)^2$ in the following.

Following~\cite{Bourdin-Francfort-EtAl-2008b}, at each time step $t_i$ $(0<i\le N)$, we seek a pair $(u_i,\alpha_i)$ solution of
\begin{equation}
   \label{eq:min problem}
      (u_i,\alpha_i) = \mathop{\mathrm{argmin}} \limits_{\substack{(v,\beta) \in V_i \\ \alpha_{i-1}  \le \beta \le 1}} \mathcal{E}_\ell(v,\beta),
\end{equation}
with the convention that $\alpha_0 = 0$.
Note that a simple truncation argument shows that the upper bound $\alpha \le 1$ is naturally satisfied by the global minimizers of $\mathcal{E}_\ell$ and does not have to be explicitly enforced.
At each step $t_i$, $i>0$, the Euler-Lagrange first order necessary conditions for optimality consist in searching for $u_{i} \in U_i$ such that
\begin{subequations}
   %\label{eq:continuous Euler-Lagrange}
   \begin{equation}
      \label{eq:continuous balance disp}
      \int_{\Omega} (a(\alpha_i) + \eta_\ell) \mathbb{C}\e(u_{i}) \cdot \e(v) \, dx= 0
   \end{equation}
   for any $v \in \left(H^1(\Omega)\right)^d$ such that $v = 0$ on $\partial\Omega_D$ and for $\alpha_{i} \in A$, $\alpha_{i-1} \le \alpha_{i}$ a.e. such that 
\begin{multline}
\label{eq:continuous inequality}
 \int_{\Omega} \mathbb{C}\e(u_i) \cdot \e(u_i) \frac12a'(\alpha_i) (\beta - \alpha_{i})  \, dx \\
      + \frac{G_c}{c_w} \int_{\Omega}
      \left(2\ell \nabla \alpha_{i} \cdot \nabla (\beta - \alpha_{i}) + \frac{w'(\alpha_i)}{\ell}(\alpha_i - \beta) \right)\, dx \ge 0,
      \end{multline}
for any $\beta \in K_i:= \left\{\beta \in H^1(\Omega) ;\ \beta \ge \alpha_{i}\ a.e.,\ \beta = 0 \text{ on } \partial \Omega_D\right\}$.
\end{subequations}

%\textbf{Shouldn't this be $\int_{\Omega} \mathbb{C}\e(u_i) \cdot \e(u_i) (\alpha -1)(\beta-\alpha_{i-1})  \, dx$ ?}

It is now well-known that as $\ell \to 0$, the phase-field energy $\mathcal{E}_\ell$ $\Gamma$-converges to Francfort and Marigo's generalized Griffith energy 
\begin{equation}
   \mathcal{E}(u,\Gamma) := \frac12\int_{\Omega\setminus \Gamma} \mathbb{C}\e(u)  \cdot \e(u) \, dx + G_c \mathcal{H}^{d-1}(\Gamma\cap \bar{\Omega}),
   \label{eq:defFM}   
\end{equation}
so that the solutions of~\eqref{eq:min problem} converge to unilateral minimizers of $\mathcal{E}$ subject to a crack growth hypothesis, and that the set $\left\{\alpha(x) > 0 \right\}$ converges in some sense to the crack $\Gamma$ in~\eqref{eq:defFM} (see~\cite{Giacomini-2005,Bourdin-Francfort-EtAl-2008b,DalMaso-Iurlano-2013a}, for instance).

% \bibliography{bib}
% \end{document}

\section{Space discretization}
\label{sec:discrete setting}
For a given mesh $\mathcal{T}$, the diameter of the mesh is defined as $h := \max_{c \in \mathcal{T}} h_c$, where $h_c$ is the diameter of the cell $c$.
We consider a sequence of matching simplicial meshes $(\mathcal{T}_h)_h$ indexed by $h$ with $h \to 0$.
The mesh is supposed to be shape regular in the sense of~\cite{ciarlet2002finite}, \emph{i.e.} there exists a parameter $\rho > 0$, independent of $h$, such that, for all $c \in \mathcal{T}_h$, $\rho h_c \le r_c$, where $r_c$ is the radius of the largest ball inscribed in $c$.
Let $\mathbb{P}^1_d(\mathcal{T}_h)$ be the set of broken polynomials of order one and dimension $d$ on the mesh $\mathcal{T}_h$, and $\mathbb{P}^1(\mathcal{T}_h)$ a similar set for scalar polynomials, see \cite{di2011mathematical}.
The set of mesh facets of the mesh $\mathcal{T}_h$ is written $\mathcal{F}_h$ and is partitioned into $\mathcal{F}_h := \mathcal{F}^i_h \cup \mathcal{F}^b_h$, where $\mathcal{F}^i_h$ is the set on internal facets and $\mathcal{F}^b_h$ is the set of boundary facets, thus $\forall F \in \mathcal{F}^b_h, F \subset \partial \Omega$.
%The set of boundary facets is partitioned into $\mathcal{F}^b_h := \mathcal{F}^b_{hD} \cup \mathcal{F}^b_{hN}$, where $\forall F \in \mathcal{F}^b_{hD}, F \subset \partial \Omega_D$ and $\forall F \in \mathcal{F}^b_{hN}, F \subset \partial \Omega_N$.
The broken gradient $\nabla_h$ is defined for a function $f$ as $\nabla_h f = \sum_{c \in \mathcal{T}_h} \nabla (f_{|c})$, where $f_{|c}$ is the restriction of a function $f$ to a cell $c$.
The discrete strain $\e_h$ is then defined as the symmetric part of $\nabla_h$.
For an inner facet $F \in \mathcal{F}^i_h$, we define $c_{F,+},c_{F,-} \in \mathcal{T}_h$ to be the cells sharing the facet $F$.
A normal vector to each inner facet $F$, written $n_F$, is directed from $c_{F,-}$ towards $c_{F,+}$.
The average and jump of a function $u_h \in \mathbb{P}^1_d(\mathcal{T}_h)$ over the inner facet $F \in \mathcal{F}^i_h$ are respectively defined as $\{u_h\}_F := \frac12(u_{c_{F,+}}+u_{c_{F,-}})$ and $[u_h]_F := u_{c_{F,+}} - u_{c_{F,-}}$, where $u_c$ is the restriction of $u_h$ to the cell $c \in \mathcal{T}_h$.
%For a boundary facet $F \in \mathcal{F}^b_h$, the average and jump of a function $u_h \in \mathbb{P}^1_d(\mathcal{T}_h)$ are defined respectively as $\{u_h\}_F = u_{c,F}$ and $\{u_h\}_F = u_{c,F}$ where $c_F \in \mathcal{T}_h$ is the only cell containing the boundary facet $F$.

\subsection{Definition of discrete spaces}
The Dirichlet boundary conditions are imposed strongly both for the displacement and the damage.
We define $\pi_h$ to be the $L^2$ projection onto $\mathbb{P}^1_d(\mathcal{T}_h)$ or $\mathbb{P}^1(\mathcal{T}_h)$, \emph{i.e.} $\pi_h f_i$ is the solution of
\[ \int_\Omega (\pi_h f _i - f_i) \cdot v_h = 0, \quad \forall v_h \in \mathbb{P}^1_d(\mathcal{T}_h). \]
Since $w_i \in \left(H^{1/2}(\partial \Omega)\right)^d$, there exist $f_i \in \left(H^1(\Omega)\right)^d$ such that ${f_i}_{|\partial \Omega_D} = w_i$ on $\partial \Omega_D$.
%\FM{As $w_i \in \left(H^{1/2}(\partial \Omega_D)\right)^d$, there exists $z_i \in \left(H^{1}(\Omega)\right)^d$ such that $w_i = z_i$ on $\partial \Omega_D$.}
Let $U_{i,h} := \{u_h \in \mathbb{P}^1_d(\mathcal{T}_h) | u_h = \pi_h f_i \text{ on } \partial \Omega_D \}$ and $A_h := \{ \alpha_h \in \mathbb{P}^1(\mathcal{T}_h) | \forall F \in \mathcal{F}^i_h, \int_F [\alpha_h]_F = 0 \text{ and } \alpha_h = 0 \text{ on } \partial \Omega_D \}$.

As a consequence, the displacement $u_h$ is cellwise $\mathbb{P}^1$ and discontinuous across facets and the damage $\alpha_h$, which is $\mathbb{P}^1$ Crouzeix--Raviart, is continuous only at the barycentre of the inner facets \cite{ern_guermond}.
Figure \ref{fig:dofs} sketches the location of the degrees of freedom (dofs) for the elements mentioned above.
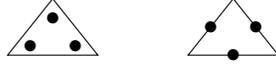
\begin{figure}[!htp]
\centering
\begin{tikzpicture}[scale=1.2]

%DG
\draw (0,0) -- (1,0) -- (0.5,0.625) -- cycle;
\node at (0.25,0.1) {$\bullet$};
\node at (0.75,0.1) {$\bullet$};
\node at (0.5,0.4) {$\bullet$};

%CR
\draw (2,0) -- (3,0) -- (2.5,0.625) -- cycle;
\node at (2.25,0.3) {$\bullet$};
\node at (2.5,0) {$\bullet$};
\node at (2.75,0.3) {$\bullet$};

\end{tikzpicture}
\caption{Dofs in a cell. Left: discontinuous Galerkin. Right: Crouzeix--Raviart.}
\label{fig:dofs}
\end{figure}
Thus the discrete space to look for a solution is $V_{i,h} := U_{i,h} \times A_h$ and the associated test space is $V_{0,h} := U_{0,h} \times A_{h}$ where functions in $U_{0,h}$ verify homogeneous Dirichlet boundary conditions on $\partial \Omega_D$.
We define the following discrete norms on these spaces to study the discrete problem.
The jump semi-norm is defined as
\begin{equation*}
%\label{eq:jump semi-norm}
   |u_h|^2_J := \sum_{F \in \mathcal{F}^i_h} \frac{1}{h_F} \Vert[u_h]_F\Vert^2_{L^2(F)}
\end{equation*}
The interior penalty norm is defined as
\[   \Vert u_h \Vert^2_{ip} := \sum_{c \in \mathcal{T}_h} \Vert \e(u_c)\Vert^2_{L^2(c)} + |u_h|^2_J. \]
   
\subsection{Naive discretization of displacements}
Let $(u_h,\alpha_h) \in V_{i,h}$ {and $v_h \in U_{0,h}$.}
Naively, one may want to discretize the bilinear form in~\eqref{eq:continuous balance disp} as
\begin{equation}
\label{eq:naive}
\int_{\Omega} (a(\alpha_h) + \eta_\ell) \mathbb{C}\e_h(u_h) \cdot \nabla_h v_h \, dx.
\end{equation}
However, unlike functions in $V_i$, the functions in $V_{i,h}$ can jump between cells so that the solution $u_i$ of~\eqref{eq:continuous balance disp} may not satisfy~\eqref{eq:naive}.
To make sure that it is the case, a consistency term has to be added to the term of~\eqref{eq:naive} to write a discrete formulation.
To build this consistency term, one uses an integration by parts on each cell and gets
\begin{multline*}
   \int_{\Omega} (a(\alpha_h) + \eta_\ell) \mathbb{C}\e_h(u_h) \cdot \nabla_h v_h \, dx 
   = -\sum_{c \in \mathcal{T}_h} \int_{c} \mathrm{div}\left((a(\alpha_h) + \eta_\ell) \mathbb{C}\e(u_h)\right) \cdot v_h \, dx \\
    + \int_{\partial c} \left((a(\alpha_h) + \eta_\ell) \mathbb{C}\e(u_h) v_h \right) \cdot n\, dS,      
\end{multline*}
so that 
\begin{multline}
\label{eq:ibp}
\int_{\Omega} (a(\alpha_h) + \eta_\ell) \mathbb{C}\e_h(u_h) \cdot \nabla_h v_h \, dx = 
 -\sum_{c \in \mathcal{T}_h} \int_{c} \mathrm{div}\left((a(\alpha_h) + \eta_\ell) \mathbb{C}\e(u_h)\right) v_h \, dx \\
 + \sum_{F \in \mathcal{F}_h^i} \int_F \left[(a(\alpha_h) + \eta_\ell) \mathbb{C}\e(u_h) v_h \right]_F \cdot n_F\, dS.
\end{multline}
Testing for consistency by replacing $u_h$ with $u$, on any inner facet $F \in \mathcal{F}_h^i$, the last term in the right-hand side of~\eqref{eq:ibp} becomes 
\begin{multline}
\label{eq:pb} 
   \int_F \left[(a(\alpha_h) + \eta_\ell) \mathbb{C}\e(u)  v_h \right]_F \cdot n_F\, dS = \int_F \bigl( \left[(a(\alpha_h) + \eta_\ell) \mathbb{C}\e(u) \right]_F  \{v_h\}_F\,  \\
   +\left\{(a(\alpha_h) + \eta_\ell) \mathbb{C}\e(u) \right\}_F \cdot [v_h]_F \bigr)\cdot n_F\, dS.
\end{multline}
The last term in the right-hand side of~\eqref{eq:pb} is usual in a discontinuous Galerkin framework, see \cite{arnold1982interior}, for instance.
However, the first term in the right-hand side does not vanish since the Crouzeix--Raviart finite element is not globally continuous on an inner facet $F \in \mathcal{F}_h^i$ and $a$ is not an affine function.
Indeed, for an inner facet $F \in \mathcal{F}_h^i$, one has $\int_F [\alpha_h]_F dS = 0$ and thus one would have $\int_F [f(\alpha_h)]_F dS = 0$ for an affine function $f$.
%\FM{Therefore, using the usual consistency terms as in \cite{arnold1982interior} is not sufficient to obtain consistency and would lead to approximating the wrong continuous problem.}

%A possibility to have a vanishing term would be to replace $a(\alpha_h)$ in~\eqref{eq:naive} by $\mathcal{I}_{CR} a(\alpha_h)$ where $\mathcal{I}_{CR}$ is the $\mathbb{P}^1$ Crouzeix--Raviart interpolation operator.
%%%Because the choice of the stiffness interpolation function $a$ is known to have a strong impact on the behavior of the solutions of~\eqref{eq:min problem} (see~\cite{marigo2016overview} for instance), we propose a different solution.
%However, that is the path chosen in this paper.
%\TODO{We thus present subsequently a different discretization to avoid having such a term.}
%\TODO{This is not true. What we do works because we get in the context of heterogeneous diffusion but that's all.}
%\TODO{Using a $\mathcal{I}_{CR}$ interpolation would kill the term but the locally small values of $\alpha$ push us in that direction...}
%\BBFIXME{Aren;t you suggesting to project $a(\alpha)$ on discontinuous polynomials of order 0 instead of 1st order CR? Why do you think that it is better? How can we justify this decision?}
\subsection{Discretization of displacement evolution}
Instead, we acknowledge the existence of jumps of $a(\alpha_h)$ and will instead consider the material studied as heterogeneous in the sense that $\mathbb{C}(\alpha) := (a(\alpha) + \eta_\ell) \mathbb{C}$.
A method to handle heterogeneous materials was introduced in~\cite{dryja2003discontinuous} and analysed in~\cite{di2008discontinuous}.
However, it requires to have cellwise constant material parameters.
Thus, the $a(\alpha_h)$ in~\eqref{eq:naive} is projected onto the set $\mathbb{P}^0(\mathcal{T}_h)$ of cellwise constant functions.
We define $\Pi_h$ as the $L^2$-orthogonal projection from $L^2(\Omega)$ onto $\mathbb{P}^0(\mathcal{T}_h)$ and thus for $\phi \in L^2(\Omega)$, one has
\begin{equation*}
%\label{eq:projection}
\Pi_h \phi := \sum_{c \in \mathcal{T}_h} { \left( \frac1{|c|} \int_c \phi dx \right) \chi_c,}
\end{equation*}
{where $\chi_c$ is the indicator function of the cell $c \in \mathcal{T}_h$.}
We thus define $a_h(\alpha_h) := \Pi_h a(\alpha_h)$, where $a(\alpha_h)\in L^2(\Omega)$.
Note that $a_h(\alpha_h) {\ge 0}$, by definition of $a$.
Following~\cite{di2008discontinuous}, we define the weighted average over an inner facet $F \in \mathcal{F}^i_h$ as, for $u_h \in U_{i,h}$,
\begin{equation*}
   \{\sigma_h(u_h) \cdot n_F\}_{w,F} :=  \frac{\left\langle \left(a_h(\alpha_h) +\eta_\ell\right) \sigma_h(u_h) \cdot n_F \right\rangle_F}{\left<a_h(\alpha_h) + \eta_\ell\right>_F },
\end{equation*}
where $\left<\alpha_h\right>_F := \frac12 \left(\alpha_{c_{F,-}} + \alpha_{c_{F,+}} \right)$ is the average value of $\alpha_h$ over an inner facet $F \in \mathcal{F}^i_h$.
The corresponding consistency term is 
\begin{equation*}
-\sum_{F \in \mathcal{F}_h^i} \int_F \{(a_h(\alpha_h)+\eta_\ell) \sigma_h(u_h) \cdot n_F \}_{w,F} \cdot [v_h]_F dS,
\end{equation*}
where $u_h \in U_{i,h}$ and $v_h \in U_{0,h}$.
It is used to have the strong consistency of the discrete bilinear form $\mathcal{U}_h$ defined below in Equation~\eqref{eq:discrete bilinear form} in the sense that $\mathcal{U}_h(\alpha_h;u,v_h) = 0$, for all $v_h \in U_{0,h}$.
Following~\cite{dryja2003discontinuous}, we define the penalty term as
\begin{equation*}
\sum_{F \in \mathcal{F}_h^i} \frac{\zeta}{h_F} \gamma_F \int_F [u_h]_F \cdot [v_h]_F dS,
\end{equation*}
where $\zeta > 0$ is a penalty parameter and 
we define for an inner facet $F \in \mathcal{F}^i_h$,
\begin{equation}
\label{eq:weighted pen}
\gamma_F := \frac{(a_h(\alpha_h)_{c_{F,+}}+\eta_\ell)( a_h(\alpha_h)_{c_{F,-}}+\eta_\ell)}{\left<a_h(\alpha_h) + \eta_\ell \right>_F} > 0,
\end{equation}
where $\gamma_F$ is twice the harmonic average of $a_h(\alpha_h) +\eta_\ell$ over $F$.

Because the penalty coefficient~\eqref{eq:weighted pen} can become locally very small when $\alpha_F \to 1$, we resort to a non-symmetric discontinuous Galerkin formulation~\cite{riviere2001priori,riviere2003discontinuous}, which has the main advantage of being stable even with small penalty terms, as proved in Proposition~\ref{th:discrete solution disp} below.
Let us then write the discretization of the bilinear form in~\eqref{eq:continuous balance disp}, for $u_h \in U_{i,h}$ and $v_h \in U_{0,h}$:
\begin{multline}
   \label{eq:discrete bilinear form}
       \mathcal{U}_h(\alpha_h;u_h,v_h) := \int_{\Omega} (a_h(\alpha_h)+\eta_\ell) \mathbb{C}\e_h(u_h) \cdot \e_h(v_h)dx \\
       - \sum_{F \in \mathcal{F}_h^i} \int_F \bigl(\{(a_h(\alpha_h)+\eta_\ell) \sigma_h(u_h) \cdot n_F \}_{w,F} \cdot [v_h]_F \\
       - \{(a_h(\alpha_h)+\eta_\ell) \sigma_h(v_h) \cdot n_F \}_{w,F} \cdot [u_h]_F \bigr) dS \\
       + \sum_{F \in \mathcal{F}_h^i} \frac{\zeta}{h_F} \gamma_F \int_F [u_h]_F \cdot [v_h]_F dS,
\end{multline}
As a consequence, the discretized Euler--Lagrange equation consists in searching for $u_h \in U_{i,h}$ such that
\begin{equation}
   \label{eq:discrete eq u 2}
   \mathcal{U}_h(\alpha_h;u_h,v_h) = 0, \quad \forall v_h \in U_{0,h}.
\end{equation}
Equation~\eqref{eq:discrete eq u 2} admits a unique solution as proved in Proposition \ref{th:discrete solution disp}.

\subsection{Discretization of damage evolution}
The second Euler--Lagrange equation is actually a variational inequality due to the irreversibility constraint $\alpha_{i-1,h} \leq \alpha_h$, where $\alpha_{i-1,h}$ is the value of the damage variable at the previous time step $t_{i-1}$.
The associated bilinear form for any $\beta_h \in A_h$ is
\begin{equation*}
   \mathcal{A}_{h}(u_h; \alpha_h, \beta_h) = \int_{\Omega} \mathbb{C}\e_h(u_h) \cdot \e_h(u_h) \alpha_h \beta_h dx  \\
%    - \sum_{F \in \mathcal{F}_h^i} \int_F n \cdot  \left(\{2\alpha_h \beta_h \sigma_h(u_h) \}_F \cdot [u_h]_F \right) dS \\
%    + \sum_{F \in \mathcal{F}_h^i} \frac{\zeta}{h_F} \int_F |[u_h]_F|^2 \{\alpha_h\}_F \{\beta_h\}_F dS
    + \frac{2G_c}{c_w} \int_{\Omega} \ell \nabla_h \alpha_h \cdot \nabla_h \beta_h dx,
\end{equation*}
and the linear form is 
\[ f(u_h; \beta_h) := \int_\Omega \mathbb{C}\e_h(u_h) \cdot \e_h(u_h) \beta_h dx + \frac{G_c}{c_w} \int_{\Omega} \frac{\beta_h}{\ell} dx. \]
Thus, we introduce the cone $K_{i,h}:=\{\beta_h \in A_{h} | \alpha_{i-1,h} \le \beta_h\}$ and the Euler--Lagrange equation becomes: search for $\alpha_h \in K_{i,h}$,
\begin{equation}
   \label{eq:discrete inequality}
   \mathcal{A}_{h}(u_h; \alpha_h, \beta_h - \alpha_{i,h}) \ge f(u_h; \beta_h - \alpha_{i,h}), \quad \forall \beta_h \in K_{i,h}.
\end{equation}
Equation~\eqref{eq:discrete inequality} admits a unique solution as proved in Proposition~\ref{th:damage ineq}.

\subsection{Well posedness and convergence of the scheme}
The regularization length $\ell$ is kept constant in this proof.
We only explore the effect of having the mesh size $h \to 0$.

\begin{proposition}~\eqref{eq:discrete eq u 2} admits a unique solution.
\label{th:discrete solution disp}
\end{proposition}

\begin{proof}
As $U_{i,h}$ is a finite-dimensional space, uniqueness and existence of a solution $u_{i,h}$ to~\eqref{eq:discrete eq u 2} are equivalent.
Let us prove uniqueness by considering $u_h \in U_{0,h}$ solution of~\eqref{eq:discrete eq u 2}.
We want to prove that $u_h = 0$ in $\Omega$.
One has 
\[ \mathcal{U}_h(\alpha_h;u_h,u_h) = 0, \]
and thus 
\[ \eta_\ell 2\mu \Vert \e_h(u_h) \Vert_{L^2(\Omega)}^2 + \frac{\zeta \eta_\ell^2}{1 + \eta_\ell}|u_h|^2_J \le 0. \]
Therefore, $[u_h]_F = 0$, for all $F \in \mathcal{F}^i_h$, and $u_h$ is continuous in $\Omega$.
We deduce that $0 = \e_h(u_h) = \e(u_h)$ in $\Omega$.
Thus, $u_h$ is a constant rigid body motion in $\Omega$.
However, taking into account the homogeneous Dirichlet boundary conditions in $U_{0,h}$, $u_h=0$ in $\Omega$.
\end{proof}

\begin{remark}[Non-variational discretization]
Because the bilinear form that is the second term in the left-hand side of~\eqref{eq:discrete bilinear form} is skew-symmetric, its curl in the second variable ($u_h$) is not zero.
As a consequence, this term cannot be the gradient of a potential energy and thus a discrete equivalent to~\eqref{eq:energy continuous Phase-field} cannot be provided.
\end{remark}

\begin{proposition}
\label{th:damage ineq}
\eqref{eq:discrete inequality} admits a unique solution.
\end{proposition}
\begin{proof}
\eqref{eq:discrete inequality} is the variational inequality related to the following constrained minimization problem:
   \[ \min_{\alpha_{i-1,h} \le \alpha_h \le 1} \mathcal{E}_{\ell,h}(u_h,\alpha_h). \]
   $A_h$ is finite dimensional thus the existence of a solution to the unconstrained minimization problem is ensured by the fact that $\mathcal{E}_{\ell,h}$ is continuous in $\alpha_h$ and proper ($\mathcal{E}_{\ell,h}(u_h,\alpha_h) \to +\infty$ when $|\alpha_h| \to +\infty$). The solution is unique because $\mathcal{E}_{\ell,h}$ is strictly convex in $\alpha_h$.
   The irreversibility constraint is taken into account by adding to $\mathcal{E}_{\ell,h}$ the following indicator function which is proper convex lower semi-continuous
   \[\mathbf{X}_{[0,\alpha_{i-1,h}]} = \left\{
      \begin{alignedat}{3}
         &0& &\text{ if } \alpha_{i-1,h} \leq \alpha_h\\
         &+\infty& &\text{ otherwise}
      \end{alignedat}
      \right.
   \]
   Thus $\mathcal{E}_{\ell,h}(u_h,\cdot) + \mathbf{X}_{[0,\chi_h]}$ verifies the necessary hypotheses to ensure the existence and uniqueness of a solution to~\eqref{eq:discrete inequality}.
\end{proof}
The following theorem proves the convergence of the scheme when $h \to 0$.
\begin{theorem} [Convergence]
\label{th:convergence sym}
There exists a solution $(u_h, \alpha_h) \in V_{i,h}$ of~\eqref{eq:discrete eq u 2} and~\eqref{eq:discrete inequality}.
Furthermore, the sequence $(u_h, \alpha_h)_h$ converges when $h \to 0$ towards $(u_i,\alpha_i) \in U_i \times K_i$ which is a solution of~\eqref{eq:continuous balance disp} and~\eqref{eq:continuous inequality}.
\end{theorem}
For concision, the proof is postponed to Appendix \ref{sec:proof}.

\section{Numerical experiments}
\label{sec:numerical tests}
In all the following numerical experiments, the penalty parameter $\zeta$ appearing in~\eqref{eq:discrete bilinear form} is chosen as $\zeta := 2\mu$, where $\mu$ is the second Lam\'e coefficient and $\eta_\ell$ is taken to be such that $\eta_\ell = 10^{-6}$ in all the numerical tests.
The numerical implementation is based on the \texttt{C++} and \texttt{Python} library \texttt{FEniCS}~\cite{LoggEtal2012}.

\subsection{Surfing boundary conditions}
Owing to classical~\cite{Zehnder-2012a} and more modern~\cite{Chambolle-Francfort-EtAl-2009a} studies, a single crack propagating in an isotropic homogenous medium typically does so in mode-I. 
Formally, the jump of the displacement field is expected to be along a direction orthogonal to the tangent of the crack near its tip.
Consider a  polar coordinate system emanating at the crack tip along its tangent direction.
In this setting, the asymptotic behavior of the displacement field is the plane-stress mode-I $K$-dominant field $u_I$ defined by

\begin{equation}
   \label{eq:modeI}
   u_I(r,\theta) := \frac{K_I}{2\mu} \sqrt{\frac{r}{2\pi}} (\kappa - \cos(\theta)) \left( \cos\left(\frac{\theta}{2}\right), \sin\left(\frac{\theta}{2}\right) \right),
\end{equation}
where $K_I = \sqrt{EG_c}$ and $\kappa = \frac{3-\nu}{1+\nu}$.

Consider a rectangular domain $(0,W) \times(-H/2,H/2)$  occupied by a brittle material (see Figure~\ref{fig:surfingDomain}).

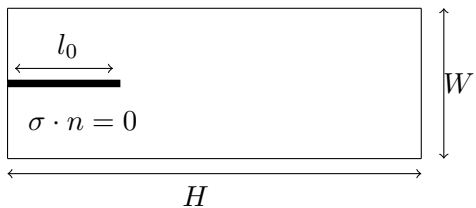
\begin{figure} [!htp]
   \centering
   \begin{tikzpicture}
   %éprouvette
   \draw (-2.5,1) -- (3,1);
   \draw (-2.5,-1) -- (3,-1);
   \draw[] (-2.5,0.) -- (-2.5,1);
   \draw (-2.5,0.) -- (-2.5,-1);
   \draw (3,1) -- (3,-1);
   \draw [<->] (-2.5,-1.2) -- (3,-1.2);
   \draw (0,-1.5) node{$H$};
   \draw[line width=1mm] (-2.5,0) -- (-1,0);
   \draw (-1.5,-0.5) node {$\sigma \cdot n = 0$};
   \draw [<->] (3.3,-1) -- (3.3,1);
   \draw (3.5,0) node{$W$};
   \draw (-1.7,0.5) node{$l_0$};
   \draw[<->] (-2.4,0.2) -- (-1.1,0.2);
   %BC droite
   %\draw (0,1.7) node[below]{$u(t_i) = w_i$};
   \end{tikzpicture}
   \caption{Surfing problem: geometry of the domain.}
   \label{fig:surfingDomain}
\end{figure} 
The idea of the surfing boundary condition, introduced in~\cite{Hossain-Hsueh-EtAl-2014a} is to prescribe the boundary displacement associated to a translating crack $\Gamma(t) = (0,Vt) \times \{0\}$ in mode-I:
\[ w_i(x,y) = U_I(x-Vt_i,y), \]
where $U_I$ denotes the expression of $u_I$ in cartesian coordinates.
The expected solution for this problem consists of a phase-field representation of $\Gamma(t)$ and associated equilibrium displacement.

%Figure~\ref{fig:surfing}(left) shows the evolution of the crack length estimated from the surface energy, the second term in~\eqref{eq:energy continuous Phase-field} for $V=4$, $2\mu = 0.77$, $\kappa = 2.08$ and three values of the regularization parameter $\ell$.

   %As expected, the normalized elastic energy release rate remains close to 1, which is consistent with Griffith's criterion $G=G_c$ for a crack propagating smoothly in space and time.
   In the classical (CG/CG) regularization, it is now well understood that the surface energy (hence the elastic energy release rate) is overestimated.
   This effect can be explained by the fact that phase-field ``cracks'' in the \ATone{} model correspond to a one-element wide strip of element where the damage field $\alpha$ takes values close to 1 and the displacement jumps sided by transition zones where the damage field decays quadratically back to 0.
   The energetic contribution of the former zone is of order $3G_ch/8\ell$ while that of the later is close to $G_c$  (see Section 8.1.1 in~\cite{Bourdin-Francfort-EtAl-2008b}).
   In contrast, in the proposed discretization, the displacement is allowed to jump along edges of internal elements so that the damage field can localize along a set of \emph{edges}. 
   Subsequently, for a given $h$ and $\ell$, our scheme is expected to provide a better approximation of the surface energy.

   Of course this could come at a cost since for a given mesh, both $\mathbb{P}^1$-discontinuous and Crouzeix-Raviart elements lead to a larger number of degrees of freedom compared to standard $\mathbb{P}^1$-Lagrange elements. 
   In Figure~\ref{fig:surfing}, we performed surfing simulations on a domain of width $W=5$ and height $H=1$ with a initial crack of length $0.3$ using the classical and proposed discretization.
   In order to obtain a comparable number of degrees of freedom for each problem, we set the mesh size to $h \simeq 0.0048 $ (leading to $2,069,307$ degrees of freedom) for the CG/CG scheme and $h\simeq 0.011$ ($2,028,499$ degrees of freedom) for the proposed scheme.
   The non-dimensional material properties are the same as in Fig.1 of~\cite{tanne2018crack}, that is $E = 1$, $\nu = 0.3$ and $G_c = 1.5$ in plane stress, and the loading rate is $V=4$.
   In all cases, after a loading stage, a phase-field crack propagates at a constant rate.
   The normalized crack ``velocity''  $\frac{\partial}{\partial t} \frac{\mathcal{E}^{(s)}_\ell(u,\alpha)}{G_cV}$, computed by linear fitting of the surface energy as a function of a loading is shown in Table~\ref{tab:crack speeds}.
   We see that even for a comparable number of degrees of freedom and despite a larger element size, (hence ratio $\ell/h$) our scheme leads to a much better approximation of the surface energy.

\begin{table}
   \centering
   \begin{tabular}{|c|c|c|}
   \hline
   scheme & $\ell/h$ & crack ``velocity''  \\ \hline
   CG/CG & 2 &  1.123 \\ \hline
   CG/CG & 5 &  1.049 \\ \hline
   DG/CR & 2 &  1.075 \\ \hline
   DG/CR & 5 &  1.036 \\ \hline
   \end{tabular}
   \caption{Surfing problem: Crack ``velocity''.}
   \label{tab:crack speeds}
\end{table}
%
   %A more fair comparison, shown below is therefore to compare the approximation of the surface energy for a comparable number of degrees of freedom.

   Figure~\ref{fig:surfing} shows the normalized elastic energy release rate $G := -\frac{1}{G_c} \partial \mathcal{P} / \partial l$, where $\mathcal{P}$ denotes the elastic energy and $l$ the crack length, computed with the $G\!-\!\theta$ method~\cite{destuynder1981interpretation,sicsic2013gradient,li2016gradient}, as a function of the loading parameter for the classical and proposed schemes for the same values of the ratio $\ell/h$.
   Again, we see that our discretization scheme outperforms the standard CG/CG scheme even for a comparable number of degrees of freedom.
      
\begin{figure}
   \centering
   \includegraphics[width=.49\textwidth]{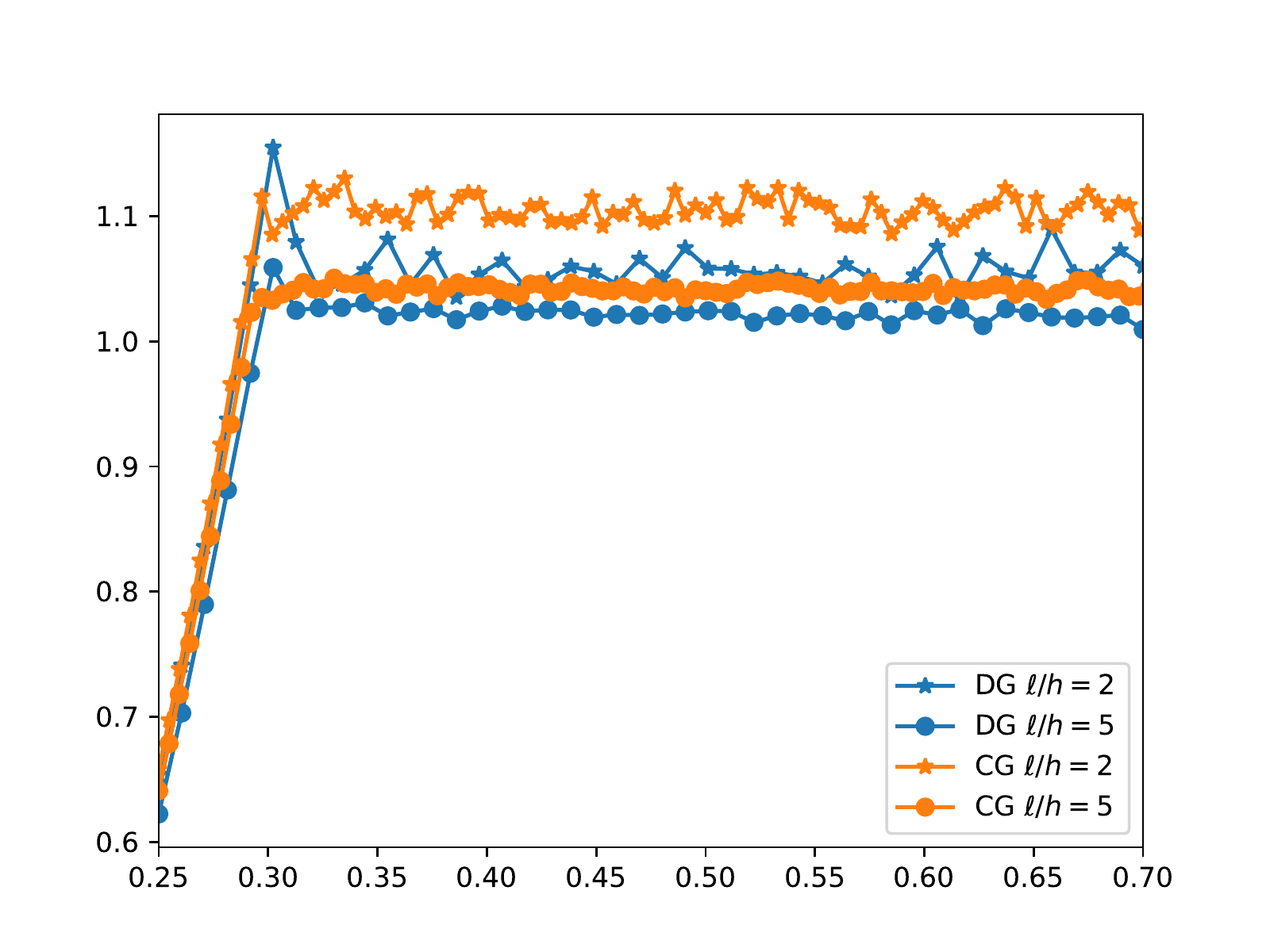}
   \caption{Surfing problem: Evolution of the energy release rate.}
   \label{fig:surfing}
\end{figure}

\subsection{Three-point bending test}
We present numerical simulations of a V-notched three-point bending test, using the specific geometry and loading from~\cite{ambati2015review}, shown in Figure~\ref{fig:bending geometry}.
\begin{figure}[!htp]
   \centering
   \includegraphics[width=.5\textwidth]{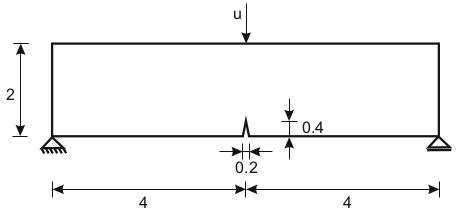}
   \caption{Three-point bending test: geometry from~\cite{ambati2015review}.}
   \label{fig:bending geometry}
\end{figure}
The vertical displacement $w_i(y) = (0,-t_i)$ is prescribed on a region of width $\SI{4.0}{\milli\meter}$ centered on the  upper edge of the domain. 
The lower-left corner of the sample is clamped and the vertical displacement of the lower-right corner has its vertical displacement blocked, while rest of the boundary is left stress-free ($\sigma \cdot n = 0$). 
The boundary value of the damage variable to 1 along both sides of the V-notch, following the guidelines of~\cite{tanne2018crack}.
The material properties are $\lambda = \SI{12.0}{\kilo\newton\per\square\milli\meter}$, $\mu = \SI{8.0}{\kilo\newton\per\square\milli\meter}$ (corresponding to $E = \SI{45}{\kilo\newton\per\square\milli\meter}$ and $\nu = 0.25$ in plane-stress conditions) and $G_c = \SI{5.4d-4}{\kilo\newton\per\milli\meter}$.
The mesh is uniform and has a size $h=\SI{0.016}{\milli\meter}$ and the regularization length is $\ell = 2h$.
The computation uses $2,348,592$ dofs.
Increments in the Dirichlet boundary conditions are chosen as $\Delta u_D = \SI{d-3}{\milli\meter}$.
Figure~\ref{fig:beam} shows the crack pattern and the load-displacement curve.
\begin{figure}
   \centering
%\subfloat{
\includegraphics[width=0.49\textwidth]{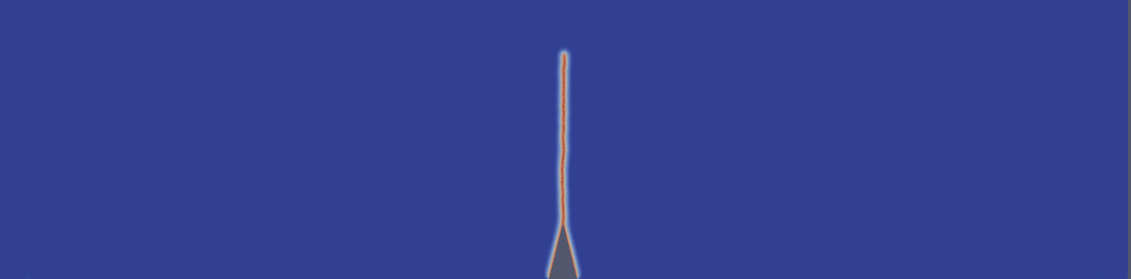}
%} 
%\subfloat{
\includegraphics[width=0.49\textwidth]{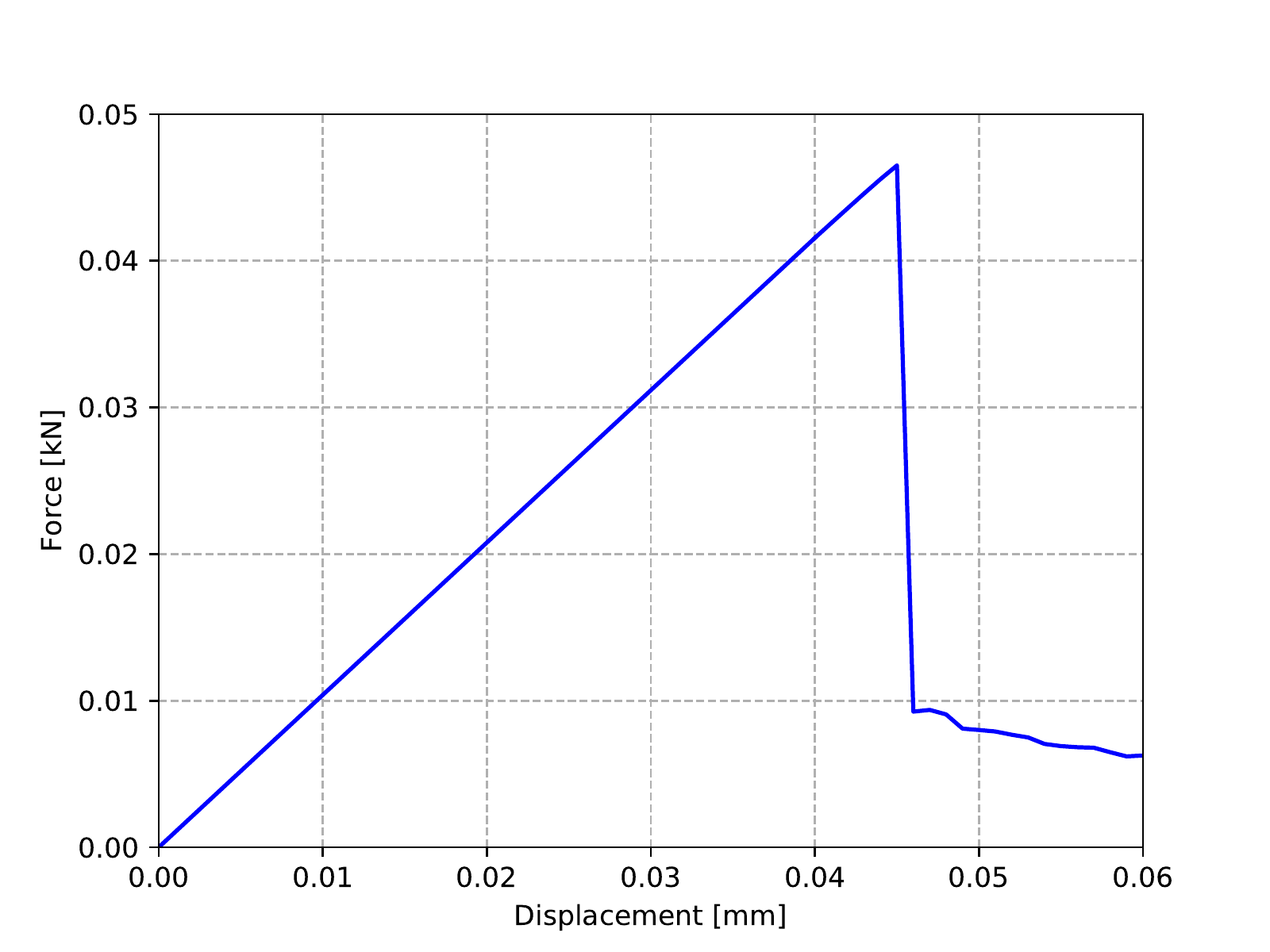}
%}
\caption{Three-point bending test. Left: crack path. Right: load-displacement curve.}
\label{fig:beam}
\end{figure}
The crack path and the load displacement curve are consistent with the results reported in~\cite{ambati2015review}, with a peak load of \SI{4.65e-2}{\kilo\newton} attained for a displacement load of \SI{4.50}{\milli \meter}.

\subsection{Mixed-mode Compact-Tension test}
Again, we repeat the variant of a compact-tension test initially presented in~\cite{ambati2015review} and repeated in~\cite{muixi2020hybridizable}.
A Dirichlet boundary condition $w_i(y) = (0,t_i)$ is imposed in the upper hole on the left of the sample whereas the lower hole on the left of the sample is clamped.
The rest of the boundary of the domain is left stress-free $\sigma \cdot n = 0$, and again, the value of the damage parameter is prescribed to 1 along the initial crack faces.
The geometry and loading are shown in Figure~\ref{fig:CT geometry}.
\begin{figure}[!htp]
   \centering
   \includegraphics[width=.5\textwidth]{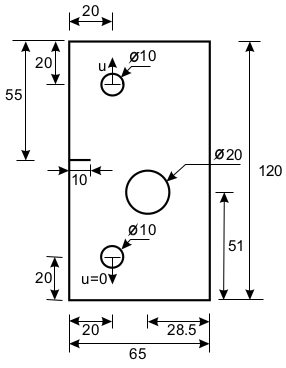}
   \caption{Notched plate with a hole. geometry from~\cite{ambati2015review}.}
   \label{fig:CT geometry}
\end{figure}
The material properties are $\lambda = \SI{1.94}{\kilo\newton\per\square\milli\meter}$, $\mu = \SI{2.45}{\kilo\newton\per\square\milli\meter}$ and $G_c = \SI{2.28d-3}{\kilo\newton\per\milli\meter}$.
%The crack has an initial position $(x,y) = (\SI{13.47}{\milli\meter},\SI{0.35}{\milli\meter})$ from the tip of the U-notch.
The mesh is uniform and has a size $h=\SI{0.39}{\milli\meter}$ and the regularization length is $\ell = \SI{0.78}{\milli\meter}$.
The computation uses $2,348,592$ dofs.
Increments in the Dirichlet boundary conditions are chosen as $\Delta u_D = \SI{d-3}{\milli\meter}$.
The offset location of the hole leads to a mixed loading mode ahead of the initial crack tip, producing a curved crack path connecting to the hole.

The loading upon which the crack nucleates from the hole and reaches for the free edge is strongly $\ell$--dependent, as expected from the analysis of~\cite{tanne2018crack}.
This effect is studied below and we focus therefore on the crack path between the initial crack and the offset hole.
Figure~\ref{fig:CT}(left,center) compares the crack path obtained using the DG/CR scheme and that obtained using a standard CG/CG discretization with the open-source code mef90/vDef~\cite{mef90} with a mesh size $h=0.15$ and regularization length $\ell=0.6$.
Both paths are essentially identical and also match that of~\cite{ambati2015review}.
The load-displacement curves for the DC/CR and CG/CG calculations, shown in Figure~\ref{fig:CT}(right) are also in strong agreement.
Conclusive comparison with a carefully validated code, give credentials to the proposed scheme.
\begin{figure}
   \centering
   \includegraphics[width=0.23\textwidth]{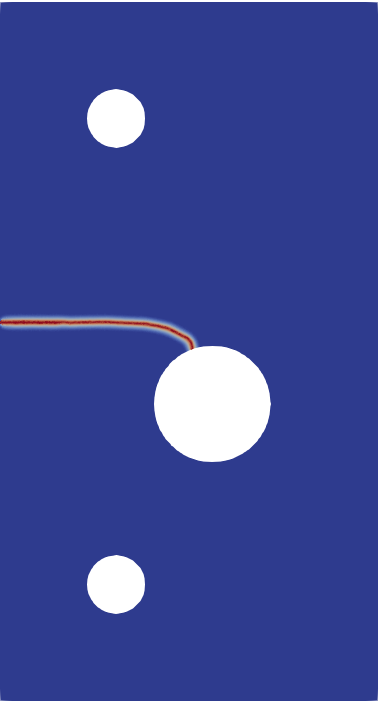}
   \includegraphics[width=0.23\textwidth]{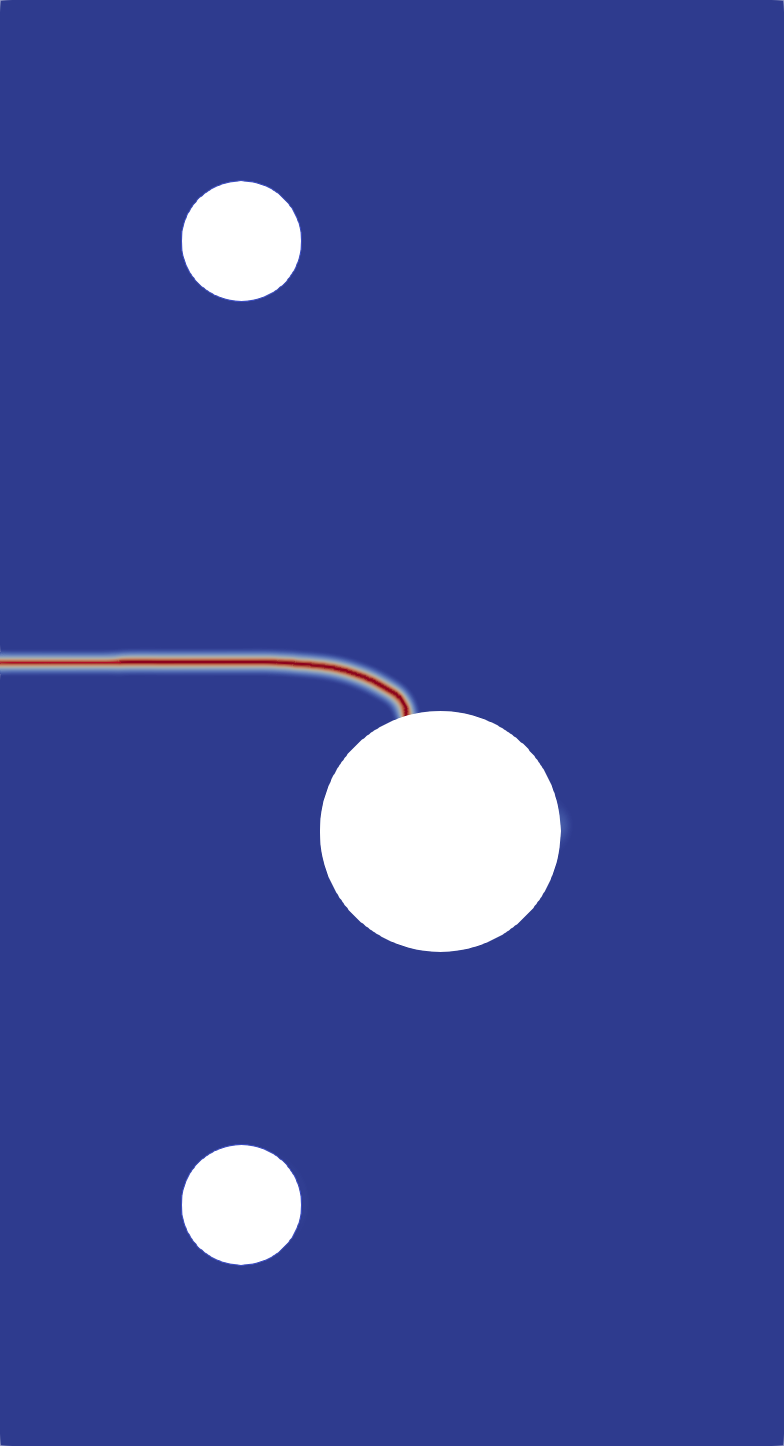}
   \includegraphics[width=0.5\textwidth]{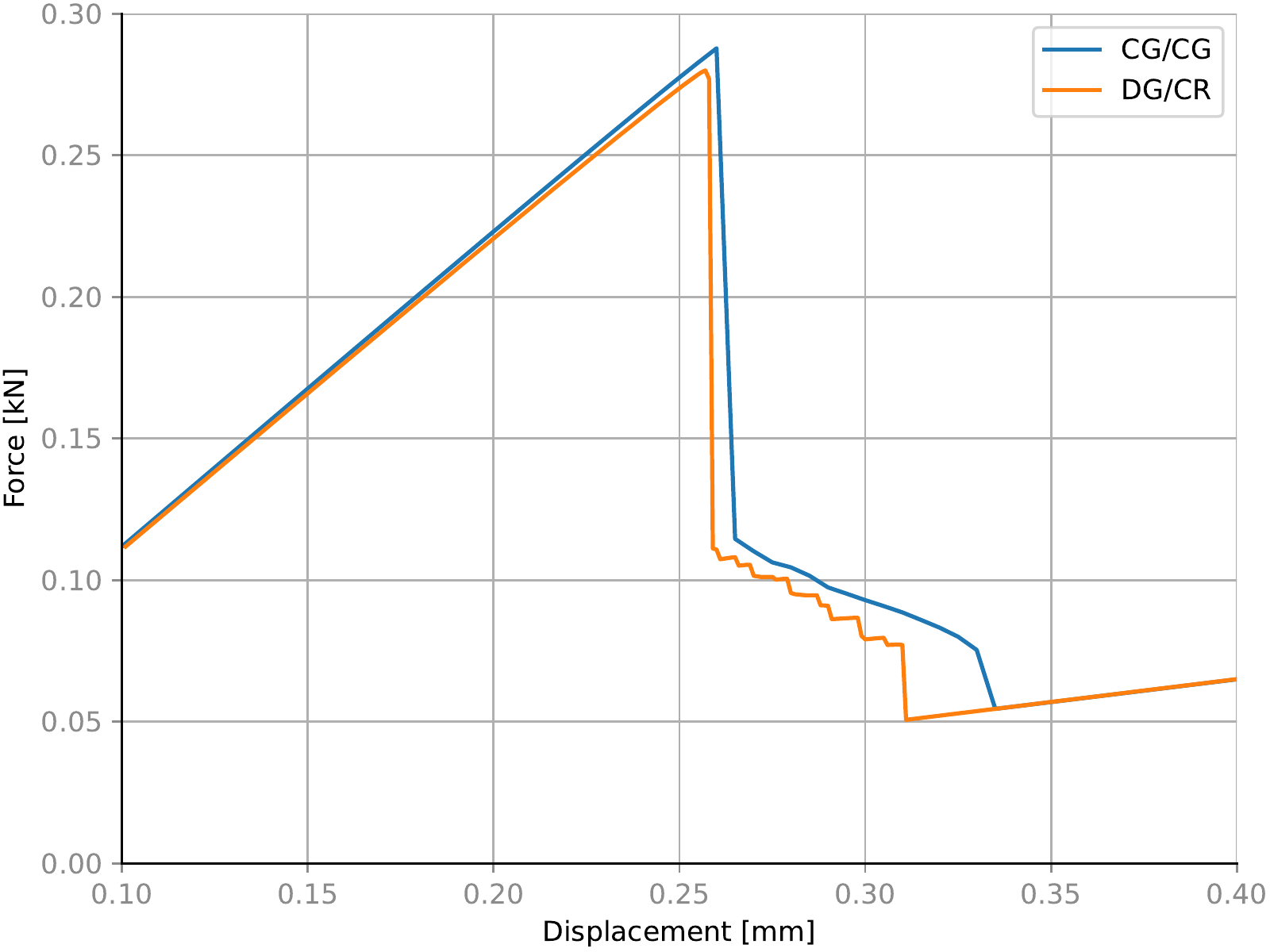} %CTFM-h0.15_out.Force.pdf}
   \caption{Notched plate with a hole from~\cite{ambati2015review}. Left: crack path using the proposed approach Center: crack path using a CG/CG discretization. Right: load-displacement curves for the DG/CR and CG/CG schemes.} %\textbf{Update fig with DG/CR LD and latest CG/CG computation. Get a DG/CR plot before renucleation}}
   \label{fig:CT}
\end{figure}

\subsection{One-dimensional nucleation test}
\label{sec:traction test}
The link between regularization length and nucleation stress has become an important feature of phase-field models of fracture, when predicting crack nucleation~\cite{tanne2018crack}.
The main ingredient in establishing this link is the study of a one-dimensional bar under uniaxial tension
as studied in~\cite{pham2011issues} (continuous setting) and~\cite{baldelli2021numerical} (discrete setting).
In these articles, a purely elastic evolution is observed for small-enough loadings.
A critical load upon which the elastic configuration becomes unstable, leading to the nucleation of a single fully-developed crack (a \emph{stable} critical point of the phase-field energy) can be computed in closed-form.
%first and then an unstable evolution in the form of crack nucleation is observed.
In~\cite{baldelli2021numerical}, the authors have developed a method to study the stability of an evolution with a classical CG/CG discretization.

Because of the loss of the variational structure of our discrete problem, it is not clear that the outcome of the stability analysis on the continuous problem provides any insight on the nucleation properties of the proposed model.
We performed series of numerical simulations to establish that it is the case.
% The stability analysis performed in both the continuous and discrete setting rely on the variational nature of the related problem.
% Unfortunately, due to the non-variational nature of our method, a similar analysis of stable states cannot be performed.
% We thus resort to numerically observing nucleation in the form of a non-zero damage field $\alpha_h$.

The computations are performed under a plane stress assumption. The domain is the rectangle $(0,1) \times (-0.05,0.05)$.
The material parameters are $E = 100$, $\nu = 0$ and $G_c = 1$. The residual strength of the material is $\eta_\ell = 10^{-6}$.
A Dirichlet boundary condition $w_i = t_i$ is applied on the normal component of the displacement on the left and right edges of the beam. Homogeneous Neumann boundary conditions are enforced on the upper and lower boundaries.
The internal length is chosen as $\frac{\ell}{h} = 5$.
The values of $\ell$ are given in Figure~\ref{fig:nucleation 1d}.
Let $t_f$ be the time corresponding to nucleation for a space-continuous evolution, one has
\[ t_f := \frac{L}{E} \sqrt{\frac{3G_c}{8E\ell}}, \quad \text{which corresponds to a stress} \quad  \sigma_f := \sqrt{\frac{3G_c}{8E\ell}}. \] 
The simulation is performed over the time-interval $[0,T]$ where $T=t_f\times 1.01$.
For $t = t_f \times 0.99$, the evolution is purely elastic for all computations.
Nucleation is detected at $t = t_f \times 1.01$ as shown in Figure~\ref{fig:after cracking}.
\begin{figure}[!htp]
   \centering
   \includegraphics[width=0.5\textwidth]{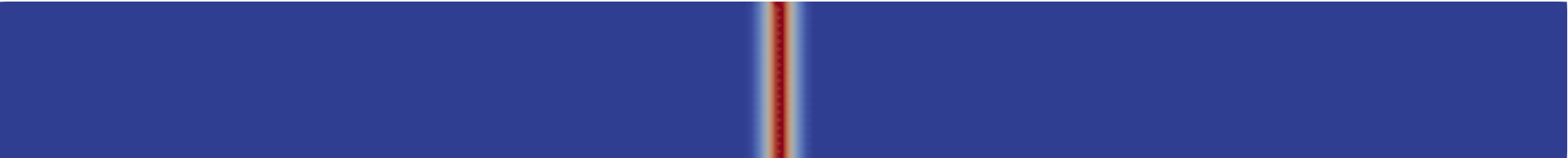}
   \caption{2d beam in uniaxial tension: $t=t_f \times 1.001$.}
\label{fig:after cracking}
\end{figure}
Figure~\ref{fig:nucleation 1d} shows the critical time computed compared with the analytical solution.

\begin{figure}[!htp]
   \centering
   \includegraphics[width=.5\textwidth]{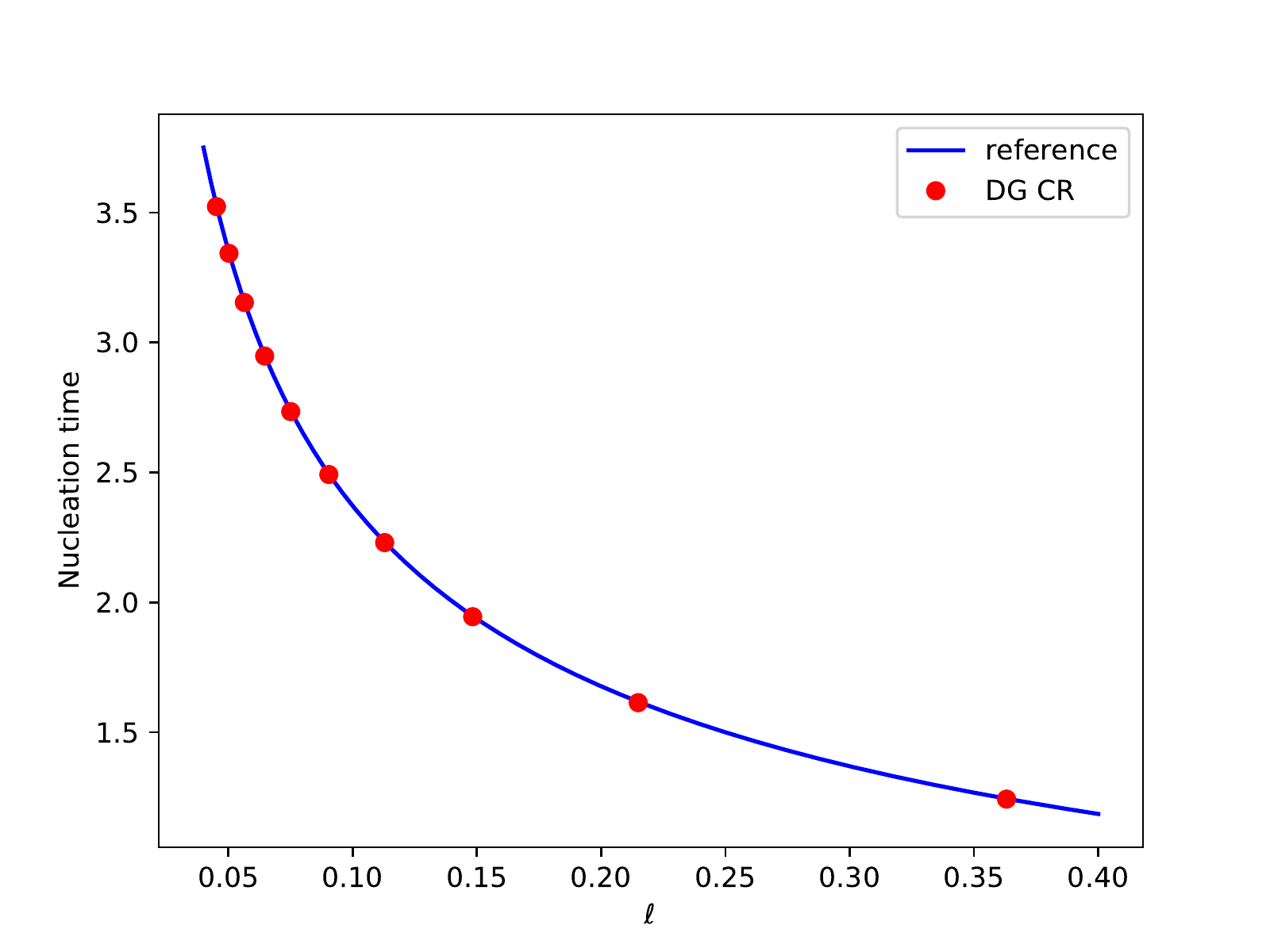}
   \caption{1d nucleation test: computed nucleation time depending on $\ell$.}
   \label{fig:nucleation 1d}
\end{figure}

\subsection{Two-dimensional nucleation test}
The behaviour of a sample under uniaxial tension has been generalized to the case of a uniformly loaded two-dimensional sample in~\cite[Appendix A]{kumar2020revisiting}.
The domain $\Omega$ is the unit disk centered at the origin. Dirichlet boundary conditions are imposed on the entire boundary as $w_i(x) = t_i \bar{E} \cdot x$, where $ \bar{E} \in \mathbb{R}^{2\times2}$ is such that 
\[ \bar{E} = \frac1{E} \left( \begin{matrix}
\cos(\theta) - \nu \sin(\theta) & 0 \\
0 & \sin(\theta) - \nu \cos(\theta) \\
\end{matrix} \right), \]
where $\theta$ is the angle in polar coordinates in $\Omega$.
Therefore, one has 
\[ \sigma_i =  t_i \left( \begin{matrix}
\cos(\theta) & 0 \\
0 & \sin(\theta) \\
\end{matrix} \right). \]
Following~\cite{kumar2020revisiting}, nucleation occurs at
\[ t_f := \sqrt{\frac{3G_cE}{8\ell (1-\nu\sin(2\theta))}}. \]
As in Section~\ref{sec:traction test}, the nucleation time is computed through a stability analysis involving the continuous energy of the system.
Thus, a similar analysis cannot be performed any longer with the proposed non-variational method.

The material parameters are $E=1$, $\nu=0.3$ and $G_c = 1.5$. The length of the phase-field is chosen as $\frac{\ell}{h} = 5$ and three meshes of sizes $h=0.04$, $h=0.02$ and $h=0.01$ are used for the test.
%Due to symmetries, only the interval $[\frac{\pi}4,\frac{3\pi}4]$ is sampled in $\theta$ so as to
We want to verify that the computed elastic domain corresponds to the theoretical one.
Figure~\ref{fig:nucleation 2d} shows the principal stresses (the eigenvalues of $\sigma_i$) depending on the sampled values of $\theta$ with respect to the analytical solution given above.
%The generalized SIF are computed from the eigenvalues of $\sigma$ as $\sigma_i \sqrt{\ell}$, where $i \in \{1,2\}$.
\begin{figure}[!htp]
   \centering
   \includegraphics[width=.5\textwidth]{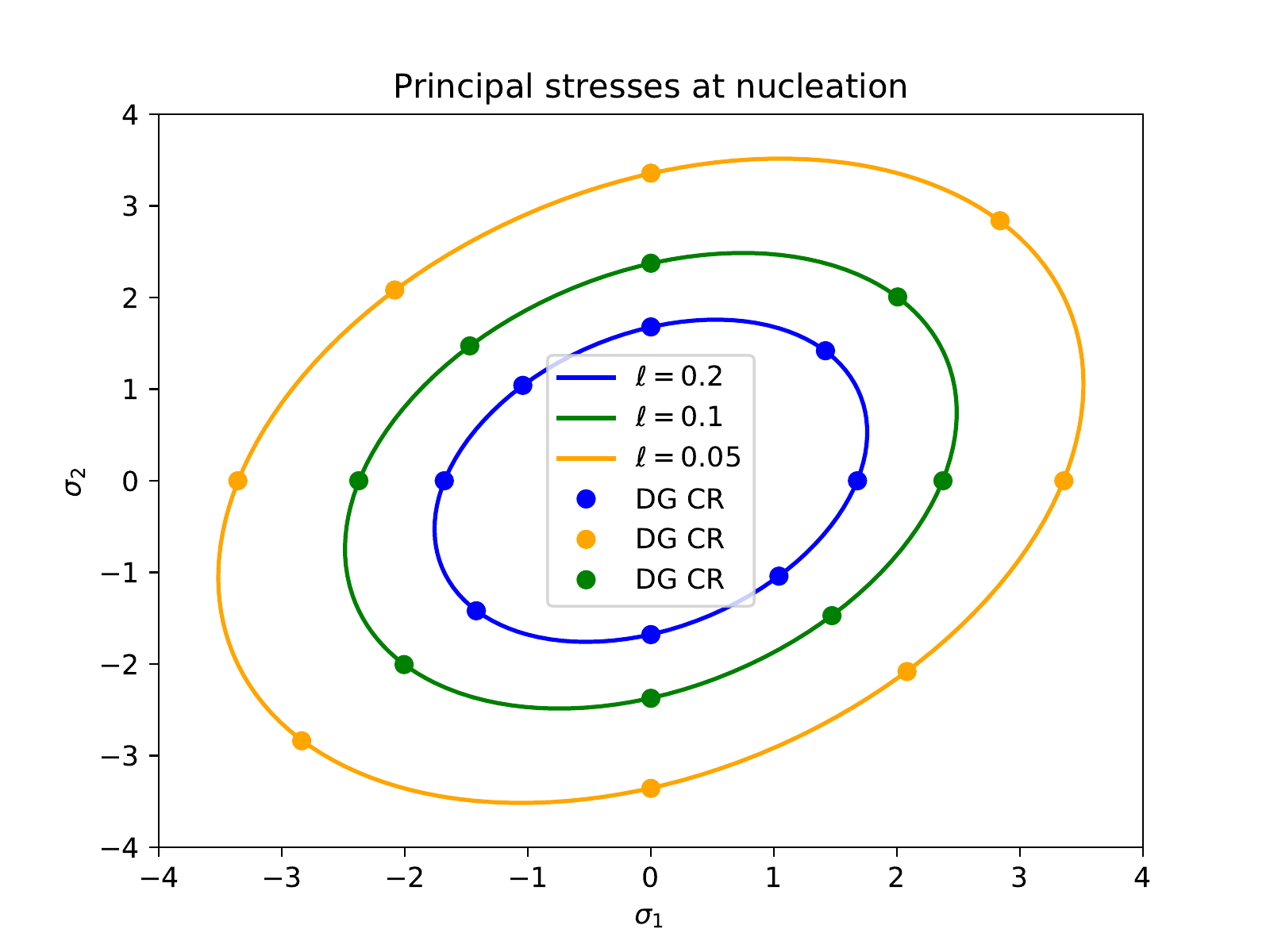}
   \caption{2d nucleation test: principal stresses at nucleation.}
   \label{fig:nucleation 2d}
\end{figure}

\section{Conclusion}
\label{sec:conclusion}
In this article, a mixed $\mathbb{P}^1$ discontinuous Galerkin and $\mathbb{P}^1$ Crouzeix--Raviart non-symmetric approximation of phase-field models for brittle fracture has been proposed in Section~\ref{sec:discrete setting}.
The non-symmetry brings improved stability compared to a symmetric method.
However, the methods loses its variational property in the sense that the discrete first order stability conditions are not discrete Euler--Lagrange equations associated to the minimization of a discrete energy.
The discretization is proved to converge towards the continuous phase-field model presented in Section~\ref{sec:continuous model}.
In Section~\ref{sec:numerical tests}, numerical evidence of the capabilities of the method regarding both crack nucleation and crack propagation are given. 
Additional investigations into the approximation of second order stability conditions for the non-symmetric method could prove valuable.

\section*{Acknowledgements}
FM would like to thank A.~Chambolle for stimulating discussions.
%This work was supported in part by the U.S. National Science Foundation grant ``Collaborative Research: Optimal Design of Responsive Materials and Structures'' (DMS 2009303).
FM's work is supported by the US National Science Foundation under grant number OIA-1946231 and the Louisiana Board of Regents for the Louisiana Materials Design Alliance (LAMDA).
Part of this work was performed while BB was the A.K. \& Shirley Barton Professor of Mathematics at Louisiana State University.
BB acknowledges the support of the Natural Sciences and Engineering Research Council of Canada (NSERC), RGPIN-2022-04536.

\section*{Code availability}
The source code and data files for all examples are available at \url{https://github.com/marazzaf/DG_CR.git}

\appendix
\section{Proof of theorem~\ref{th:convergence sym}}
\label{sec:proof}
Assume that the initial damage field $\alpha_{-1}$ is known and define $\alpha_{-1,h} := \mathcal{I}_{CR} \alpha_{-1}$. The following proof assumes that at each time-step $t_i$, $\alpha_{i-1,h}$ converges strongly in $H^1(\Omega)$ towards $\alpha_{i-1} \in H^1(\Omega)$.

\begin{lemma} [Compactness]
\label{th:compactness sym}
Let independently $u_h \in U_{i,h}$ be solution of~\eqref{eq:discrete eq u 2} and $\alpha_h \in A_h$ be a solution of~\eqref{eq:discrete inequality}.
There exists $v_{i} \in U_i$ and $\beta_i \in A$ such that, up to a subsequence, $u_h \to v_i$ strongly in $\left(L^2(\Omega)\right)^d$ and $\alpha_h \to \beta_i$ strongly in $L^2(\Omega)$, and $\nabla_h u_h \rightharpoonup \nabla v_i$ weakly in $\left(L^2(\Omega)\right)^{d \times d}$ and $\nabla_h \alpha_h \rightharpoonup \nabla \beta_i$ weakly in $L^2(\Omega)^d$.
\end{lemma}

\begin{proof}
Since $w_i \in \left(H^{1/2}(\partial \Omega)\right)^d$, there exist $f_i \in \left(H^1(\Omega)\right)^d$ such that ${f_i}_{|\partial \Omega_D} = w_i$ on $\partial \Omega_D$.
Therefore, one has
\[ \mathcal{U}_h(\alpha_h;u_h,u_h - f_i) = 0. \]
where $[f_i]_F = 0, \forall F \in \mathcal{F}^i_h$ because $f_i \in \left(H^1(\Omega)\right)^d$.
As in the proof of Proposition \ref{th:discrete solution disp}, one has
\[ \frac{2 \mu \eta_\ell}{K^2} \Vert \e_h(u_h) \Vert^2_{L^2} + \frac{\zeta \eta_\ell^2}{1+\eta_\ell} |u_h|^2_J \le \mathcal{U}_h(\alpha_h;u_h,u_h). \]
Therefore,
\[ C_1 \Vert u_h \Vert_{ip}^2 \le \mathcal{U}_h(\alpha_h;u_h,u_h) = \mathcal{U}_h(\alpha_h;u_h,f_i) \\ \le C_2 \Vert u_h \Vert_{ip}, \]
where $C_1$ and $C_2$ are non-negative constants.

Thus $\Vert u_h \Vert_{ip}$ is bounded from above.
We can apply Kolmogorov compactness criterion~\cite[p.~194]{di2011mathematical}. Thus, there exists $v_i \in U_i$ such that, up to a subsequence, $u_h \to v_i$ strongly in $\left(L^2(\Omega)\right)^d$ and $\nabla_h u_h \rightharpoonup \nabla v_i$ weakly in $\left(L^2(\Omega)\right)^{d\times d}$.

Now let us get a bound on the damage.
Testing~\eqref{eq:discrete inequality} with $\alpha_h$, one has
\[ \mathcal{A}_{h}(u_h; \alpha_h, \alpha_{i-1,h} - \alpha_h) \le f(u_h; \alpha_{i-1,h} - \alpha_h). \]
Thus, using a Cauchy--Schwarz inequality and the fact that $\alpha_{h}\leq 1$ and $\alpha_{i-1,h} \leq 1$, one has
\begin{align*}
\frac{2G_c}{c_w} & \int_{\Omega} \ell |\nabla_h \alpha_h|^2 dx \leq  \mathcal{A}_{h}(u_h; \alpha_h, \alpha_h) \leq \mathcal{A}_{h}(u_h; \alpha_h, \alpha_{i-1,h}) + f(u_h; \alpha_{i-1,h} - \alpha_h) \\
& \leq \int_{\Omega} \mathbb{C}\e_h(u_h) \cdot \e_h(u_h) dx + \frac{2G_c\ell}{c_w}\Vert\nabla_h \alpha_h\Vert_{L^2(\Omega)}\Vert\nabla_h \alpha_{i-1,h}\Vert_{L^2(\Omega)} \\
& \leq C\Vert u_h\Vert_{ip}^2 + C'\Vert\nabla_h \alpha_h\Vert_{L^2(\Omega)} \leq C + C'\Vert\nabla_h \alpha_h\Vert_{L^2(\Omega)},
\end{align*}
where $C>0$ and $C'>0$ are generic non-negative constants.

The second order polynomial in the variable $\Vert\nabla_h \alpha_h\Vert_{L^2(\Omega)}$ is negative between its two real roots and thus $\Vert\nabla_h \alpha_h\Vert_{L^2(\Omega)}$ is bounded from above.
Using the compactness of the Crouzeix--Raviart FE~\cite[p.~297]{droniou2018gradient}, there exists $\beta_i \in A$ such that, up to a subsequence, $\alpha_h \to \beta_i$ strongly in $L^2(\Omega)$ and $\nabla_h \alpha_h \rightharpoonup \nabla \beta_i$ weakly in $\left(L^2(\Omega)\right)^d$.

\end{proof}

\begin{proposition}[Existence of solution to the discretized problem]
There exists $(u_h, \alpha_h) \in V_{i,h}$ solving~\eqref{eq:discrete eq u 2} and~\eqref{eq:discrete inequality} simultaneously.
\end{proposition}

\begin{proof}
Let $T:(v_h,\beta_h) \mapsto (u_h,\alpha_h)$, where $u_h$ is the solution of $\mathcal{U}_h(\beta_h;u_h,\bullet) = 0$ over $U_{i,h}$ and $\alpha_h$ is the solution of $\mathcal{A}_h(v_h;\alpha_h,\alpha_{i-1,h}-\bullet) \le 0$ over $K_{i,h}$.
Assuming $v_h$ and $\beta_h$ verify the bounds proved in the proof of Lemma~\ref{th:compactness sym}, then $u_h$ and $\alpha_h$ verify these same bounds. Thus $T$ is a mapping of a nonempty compact convex subset of $V_{i,h}$ into itself.
As, $\mathcal{U}_h(\beta_h)$ and $\mathcal{A}_h(v_h)$ are continuous bilinear forms, $T$ is a continuous map.
As a consequence, using Brouwer fixed point theorem~\cite[p.~179]{brezis}, there exists a fixed point $(u_h,\alpha_h)$ solving~\eqref{eq:discrete eq u 2} and~\eqref{eq:discrete inequality} simultaneously.
\end{proof}

\begin{lemma}
$(v_i,\beta_i)$ is a solution of~\eqref{eq:continuous balance disp}.
\end{lemma}

\begin{proof}

Let $\varphi \in \left(\mathcal{C}_c^\infty(\Omega)\right)^d$ be a function with compact support in $\Omega$.
Testing~\eqref{eq:discrete eq u 2} with $\pi_h \varphi$, one has
\begin{multline*}
   \int_{\Omega} (a_h(\alpha_h) + \eta_\ell) \mathbb{C}\e_h(u_h) \cdot \e_h(\pi_h \varphi)dx \\
   - \sum_{F \in \mathcal{F}^i_h} \int_F n \cdot  \left(\{(a_h(\alpha_h) + \eta_\ell) \sigma_h(u_h) \}_F \cdot [\pi_h \varphi]_F -\{(a_h(\alpha_h) + \eta_\ell) \sigma_h(\pi_h \varphi) \}_F \cdot [u_h]_F \right)dS \\ + \sum_{F \in \mathcal{F}^i_h} \frac{\zeta \gamma_F}{h_F} \int_F [u_h]_F \cdot [\pi_h \varphi]_F dS = 0.
\end{multline*}
The last two terms in the left-hand side vanish when $h \to 0$ because $\varphi,v_i \in \left(H^1(\Omega)\right)^d$.
%\TODO{Modify here to get the terms $a_h(\alpha_h) - \Pi_h a(\alpha_h)$ and $\Pi_h a(\alpha_h) - \Pi_h a(\beta_i)$!!}
Regarding the first term in the left-hand side, one has
\begin{equation*}
\begin{aligned}
\int_{\Omega} (a_h(\alpha_h) + \eta_\ell) \mathbb{C}\e_h(u_h) \cdot \e_h(\pi_h \varphi)dx &= \int_{\Omega} (a(\beta_i) + \eta_\ell) \mathbb{C}\e_h(u_h) \cdot \e_h(\pi_h \varphi)dx \\
&+ \int_{\Omega} (a_h(\alpha_h) - \Pi_h a(\beta_i)) \mathbb{C}\e_h(u_h) \cdot \e_h(\pi_h \varphi)dx \\
&+ \int_{\Omega} (\Pi_h a(\beta_i) - a(\beta_i)) \mathbb{C}\e_h(u_h) \cdot \e_h(\pi_h \varphi)dx \\
& = (I) + (II) + (III)
\end{aligned}
\end{equation*}
Passing to the limit in $(I)$, one obtains the expected term
\begin{equation*}
%\label{eq:link disp damage}
   \int_{\Omega} (a(\beta_i) + \eta_\ell) \mathbb{C}\e(v_i) \cdot \e(\varphi)dx.
\end{equation*}
Let us now prove that $(II)$ and $(III)$ vanish as $h \to 0$.
Using a Cauchy--Schwarz inequality, one has
\[ \begin{aligned}
(II) &\leq \left(\int_{\Omega} (\mathbb{C}\e_h(u_h) \cdot \e_h(\pi_h \varphi))^2dx \right)^{1/2} \Vert \Pi_h a(\alpha_h) - \Pi_h a(\beta_i) \Vert_{L^2(\Omega)} \\
 &\leq C \Vert\varphi\Vert_{W^{1,\infty}(\Omega)} \Vert u_h\Vert_{ip} \Vert \Pi_h a(\alpha_h) - \Pi_h a(\beta_i) \Vert_{L^2(\Omega)}.
\end{aligned} \]
We focus on the second term in the right-hand side.
\[ %\begin{aligned}
\Vert\Pi_h a(\alpha_h) - \Pi_h a(\beta_i) \Vert_{L^2(\Omega)} \leq \Vert a(\alpha_h) - a(\beta_i) \Vert_{L^2(\Omega)},
%\\ &\leq \Vert\beta_i - \alpha_h\Vert_{L^2(\Omega)} \Vert(2+ \alpha_h + \beta_i)\Vert_{L^{\infty}(\Omega)} \\ &\leq 4 \Vert\beta_i - \alpha_h\Vert_{L^2(\Omega)},  
%\end{aligned}
\]
since $\Pi_h$ is a projection in $L^2(\Omega)$.
Using the strong convergence $\alpha_h \to \beta_i$ in $L^2(\Omega)$ and the fact that $a$ is continuous gives the desired result.
Regarding $(III)$, using a Cauchy--Schwarz inequality, one has
\[ \begin{aligned}
(III) &\leq \left(\int_{\Omega} (\mathbb{C}\e_h(u_h) \cdot \e_h(\pi_h \varphi))^2dx \right)^{1/2} \Vert a(\beta_i) - \Pi_h a(\beta_i) \Vert_{L^2(\Omega)} \\
& \le C \Vert \varphi \Vert_{W^{1,\infty}(\Omega)} \Vert u_h\Vert_{ip} \Vert a(\beta_i) - \Pi_h a(\beta_i) \Vert_{L^2(\Omega)}
\end{aligned}, \]
where $C>0$ is a generic non-negative constant.
Using a classical local approximation result (see~\cite[Proposition~1.135]{ern_guermond} for instance), one has:
\[ \Vert a(\beta_i) - \Pi_h a(\beta_i) \Vert_{L^2(\Omega)} \le Ch \Vert \nabla (a(\beta_i))\Vert_{L^2(\Omega)}, \]
where $\nabla(a(\beta_i)) = a'(\beta_i)\nabla \beta_i \in \left(L^2(\Omega)\right)^d$ because $\beta_i \in L^\infty(\Omega) \cap H^1(\Omega)$ and $a$ is $\mathcal{C}^1$ and thus $(III)$ vanishes as $h \to 0$.
%One also has $v_i = w_i$ on $\partial \Omega_D$ and thus $v_i = u_i$, the solution of~\eqref{eq:continuous balance disp}.
\end{proof}

\begin{lemma}
$\nabla_h u_h \to \nabla v_i$ strongly in $\left(L^2(\Omega)\right)^{d \times d}$, where $u_i$ is a solution of~\eqref{eq:continuous balance disp}.
\end{lemma}
\begin{proof}
We consider again $f_i \in \left(H^1(\Omega)\right)^d$ such that $f_i=w_i$ on $\partial \Omega_D$. We are going to test~\eqref{eq:discrete eq u 2} with $\tilde{v}_h = u_h - \pi_h f_i$ so that $\tilde{v}_h \in U_{0h}$.
One thus has
\[ \begin{aligned}
       & \int_{\Omega} (a_h(\alpha_h) + \eta_\ell) \mathbb{C}\e_h(u_h) \cdot \nabla_h \tilde{v}_h dx \\ &= \sum_{F \in \mathcal{F}^i_h} \int_F  n \cdot  \left(\{(a_h(\alpha_h) + \eta_\ell)\sigma_h(u_h) \}_F \cdot [\tilde{v}_h]_F - \{(a_h(\alpha_h) + \eta_\ell)\sigma_h(\tilde{v}_h) \}_F \cdot [u_h]_F \right) dS \\ &- \sum_{F \in \mathcal{F}^i_h} \frac{\zeta \gamma_F}{h_F} \int_F [u_h]_F \cdot [\tilde{v}_h]_F dS,
   \end{aligned}
\]
Using the strong convergence in $u_h$ and $\alpha_h$ in the right-hand side gives $0$ but it also ensures that the product of $e_h(u_h)$ in the left-hand side has a limit when $h \to 0$.
Thus $e_h(u_h) \to e(v_i)$ strongly in $\left(L^2(\Omega)\right)^{d\times d}$, when $h \to 0$.
Using a discrete Korn inequality over $\mathbb{P}^1_d(\mathcal{T}_h)$, we get that $\nabla_h u_h \to \nabla v_i$ strongly in $\left(L^2(\Omega)\right)^{d\times d}$.
\end{proof}

\begin{lemma} 
$(v_i,\beta_i)$ is a solution of~\eqref{eq:continuous inequality}.
\end{lemma}
\begin{proof}
We are going to test~\eqref{eq:discrete inequality} with $\beta_h = \pi_h \varphi$, where $\varphi \in \mathcal{C}^\infty_c(\Omega)$. One thus has
\begin{multline}
   \label{eq:intermediary 3}
   \int_{\Omega} \mathbb{C}\e_h(u_h) \cdot \e_h(u_h) (\alpha_h - 1) (\alpha_{i-1,h} - \pi_h \varphi) dx \\ + 2\frac{G_c}{c_w} \int_{\Omega} \ell \nabla_h \alpha_h \cdot (\nabla_h \alpha_{i-1,h} - \nabla_h \pi_h \varphi) dx \le 0.
\end{multline}
Passing to the limit $h \to 0$ in~\eqref{eq:intermediary 3}, one has
\begin{equation}
   \int_{\Omega} \e(v_i) : \mathbb{C} : \e(v_i) (\beta_i-1) (\alpha_{i-1} - \varphi) dx + 2\frac{G_c}{c_w} \int_{\Omega} \ell \nabla \beta_i \cdot (\nabla \alpha_{i-1} - \nabla \varphi) dx \le 0,
\end{equation}
using the asymptotic consistency of Crouzeix--Raviart elements~\cite[Proposition 9.5, p.286]{droniou2018gradient}.
%Because $u_i$ and $\alpha_i$ are the unique solutions to Equations~\eqref{eq:continuous balance disp} and~\eqref{eq:continuous inequality}, $(u_h)_h$ and $(\alpha_h)_h$ converge fully (not up to a subsequence) to $u_i$ and $\alpha_i$.
\end{proof}

\begin{lemma} $\nabla_h \alpha_h \to \nabla \beta_i$ strongly in $\left(L^2(\Omega)\right)^d$, where $\beta_i$ is the solution of~\eqref{eq:continuous inequality}.
\end{lemma}

\begin{proof}
One has:
\begin{multline}
   \label{eq:intermediary 4}
   \frac{2G_c\ell}{c_w} \Vert \nabla \beta_i - \nabla_h \alpha_h \Vert_{L^2}^2  \le \frac{2G_c}{c_w}\int_{\Omega} \ell( \nabla \beta_i - \nabla_h \alpha_h) \cdot (\nabla \beta_i - \nabla \alpha_{i-1})dx \\
   + \frac{2G_c}{c_w}\int_{\Omega} \ell(\nabla \beta_i - \nabla_h \alpha_h) \cdot (\nabla_h \alpha_{i-1,h} - \nabla_h \alpha_h)dx \\
    + \frac{2G_c}{c_w}\int_{\Omega} \ell(\nabla \beta_i - \nabla_h \alpha_h) \cdot (\nabla \alpha_{i-1} - \nabla_h \alpha_{i-1,h})dx.
\end{multline}
The first term in the right-hand side of~\eqref{eq:intermediary 4} goes to $0$ when $h \to 0$ because $\nabla_h \alpha_h \rightharpoonup \nabla \beta_i$ weakly in $\left(L^2(\Omega)\right)^d$. Let us now bound the second and third terms in the right-hand side of~\eqref{eq:intermediary 4}, which we write $(II)$ and $(III)$.
Using $\mathcal{A}_h(v_i;\beta_i,\alpha_{i-1,h}-\alpha_h) \le f(v_i;\alpha_{i-1,h}-\alpha_h)$ and $\mathcal{A}_h(u_h;\alpha_h,\alpha_{i-1,h}-\alpha_h) \le f(u_h;\alpha_{i-1,h}-\alpha_h)$, one has
\begin{multline}
\label{eq:intermediary 5p}
(II) \le -\int_{\Omega} \mathbb{C}\e(v_i) \cdot \e(v_i)(\beta_i-1) (\alpha_{i-1,h} - \alpha_h)dx \\
+ \int_{\Omega} \mathbb{C}\e_h(u_h) \cdot \e_h(u_h) (\alpha_h-1) (\alpha_{i-1,h} - \alpha_h)dx.
\end{multline}
Using the strong convergence $\alpha_{i-1,h} \to \alpha_{i-1}$ and $\alpha_h \to \beta_i$ in $L^2(\Omega)$, the two terms in the right-hand side of~\eqref{eq:intermediary 5p} converge to $\int_{\Omega} \mathbb{C}\e(v_i)\cdot \e(v_i) \beta_i(\alpha_{i-1} - \beta_i)$ and its opposite, thus canceling each other.
Using $\mathcal{A}_h(v_i;\beta_i,\alpha_{i-1}-\alpha_{i-1,h}) \le f(v_i;\alpha_{i-1}-\alpha_{i-1,h})$ and $\mathcal{A}_h(u_h;\alpha_h,\alpha_{i-1}-\alpha_{i-1,h}) \le f(u_h;\alpha_{i-1}-\alpha_{i-1,h})$, one has
\begin{multline}
\label{eq:intermediary 6p}
(III) \le -\int_{\Omega} \mathbb{C}\e(v_i) \cdot \e(v_i) (\beta_i-1) (\alpha_{i-1} - \alpha_{i-1,h})dx \\
+ \int_{\Omega} \mathbb{C}\e_h(u_h) \cdot \e_h(u_h) (\alpha_h-1) (\alpha_{i-1} - \alpha_{i-1,h})dx.
\end{multline}
Using the strong convergence $\alpha_{i-1,h} \to \alpha_{i-1}$ in $L^2(\Omega)$ when $h \to 0$, all the terms in the right-hand side of~\eqref{eq:intermediary 6p} vanish. Thus the strong convergence of the gradients of the damage is proved.
\end{proof}

\bibliographystyle{plain}
\bibliography{bib}

\end{document}